\documentclass[a4paper,10pt]{amsart}
\usepackage{amssymb}
\usepackage{graphicx} 
\newcounter{lixo}

\newtheorem{theorem}{Theorem}
\newtheorem{corollary}[theorem]{Corollary}

\newtheorem{proposition}[theorem]{Proposition}

\newcommand{\NN}{{\rm\bf N}}
\newcommand{\ZZ}{{\rm\bf Z}}
\newcommand{\RR}{{\rm\bf R}}
\newcommand{\EU}{{\rm\bf S}}
\newcommand{\vv}{{\rm\bf v}}
\newcommand{\ww}{{\rm\bf w}}
\newcommand{\seta}{\longrightarrow}
\newcommand{\dpt}{\displaystyle}
\newcommand{\zg}[1]{\langle\gamma_{#1}\rangle}

\begin{document}

\title[non-hyperbolic dynamics in symmetry breaking]{Partial symmetry breaking and heteroclinic  tangencies}


\author[Isabel S. Labouriau]{Isabel S. Labouriau}
\author[Alexandre A. P. Rodrigues]{Alexandre A. P. Rodrigues}

\address{Centro de Matem\'atica da Universidade do Porto\\ 
and Faculdade de Ci\^encias da Universidade do Porto\\
Rua do Campo Alegre 687, 4169--007 Porto, Portugal}

\email[ I.S.Labouriau]{islabour@fc.up.pt}
\email[A.A.P.Rodrigues]{alexandre.rodrigues@fc.up.pt}

\thanks{CMUP is supported by the European Regional Development Fund through the programme COMPETE and by the
Portuguese Government through the Funda\c{c}\~ao para a Ci\^encia e a Tecnologia (FCT) under the project PEst-C/MAT/UI0144/2011.
A.A.P. Rodrigues was supported by the grant SFRH/BPD/84709/2012 of FCT}

\begin{abstract}
We study some global aspects of the bifurcation of an equivariant family of volume-contracting vector fields on the  three-dimensional sphere.
When  part of the symmetry is broken, the vector fields exhibit Bykov cycles. 
Close to the symmetry, we  investigate the  mechanism  of the emergence of heteroclinic tangencies coexisting with transverse connections.
We find persistent suspended horseshoes  accompanied by attracting periodic trajectories with long periods. \end{abstract}

\maketitle

\section{Introduction}
Heteroclinic cycles and networks associated to equilibria, periodic solutions and chaotic sets may be responsible for intermittent dynamics in nonlinear systems. Heteroclinic cycles may also be seen as the skeleton for the understanding of complicated switching between physical states -- see Field \cite{Field}, Golubitsky and Stewart \cite{GS} and Melbourne \emph{et al} \cite{MPR}. 

The homoclinic cycle associated to a saddle-focus \cite{Homburg} provides one of the main examples for the occurrence of chaos involving suspended hyperbolic horseshoes and strange attractors; the complexity of the dynamics near these cycles has been discovered by the  pionner L. P. Shilnikov \cite{Shilnikov63, Shilnikov65, Shilnikov67A}. The simplest heteroclinic cycles between two saddle-foci of different Morse indices where one heteroclinic connection is structurally stable and the other is not have been first studied by Bykov \cite{Bykov} and are thus called \emph{Bykov cycles}. Recently there has been a renewal of interest of this type of heteroclinic bifurcation in different contexts -- see \cite{Homburg, LR, Rodrigues3} and references therein. We also refer Lamb \emph{et al} \cite{Lamb2005} who have studied Bykov cycles in the context of reversible systems.

Explicit examples of vector fields for which such cycles may be found  are reported in Aguiar \emph{et al} \cite{ACL2} and Rodrigues and Labouriau \cite{LR2}. These examples start with a differential equation with symmetry, $\dot{x}=f_0(x)$  whose flow has a globally attracting three-dimensional sphere, containing an asymptotically stable heteroclinic network with two saddle-foci. When part of the symmetry is destroyed by a small non-equivariant perturbation, it may be shown by the Melnikov method that the two-dimensional invariant manifolds intersect transversely. When some symmetry remains, the connection of the one-dimensonal manifolds is preserved, giving rise to Bykov cycles forming a network.

The main goal of this article is to describe and characterize the transition from the dynamics of the flow of the fully symmetric system $\dot{x}=f_0(x)$ and the perturbed system $\dot{x}=f_\lambda(x)$, for small $\lambda \neq 0$. For $\dot{x}=f_0(x)$, there is a heteroclinic network $\Sigma^0$ whose basin of attraction has positive Lebesgue measure.   When $\lambda\neq 0$, the intersection of the invariant manifolds is transverse giving rise to a network $\Sigma^\star$ that cannot be removed by any small smooth perturbation. The transverse intersection implies that the set of all
trajectories that lie for all time in a small neighbourhood of $\Sigma^\star$ has a locally-maximal hyperbolic set, admiting a
complete description in terms of symbolic dynamics \cite{Shilnikov67}. Labouriau and Rodrigues \cite{LR} proved that for the perturbed system $\dot{x}=f_\lambda(x)$, the flow contains a Bykov network $\Sigma^\star$ and uniformly hyperbolic horseshoes accumulating on it.

Suppose the fully symmetric network $\Sigma^0$ is asymptotically stable and let $V^0$ be a neighbourhood of $\Sigma^0$ whose closure is compact and positively flow-invariant; hence it contains the $\omega$--limit sets of all its trajectories. The union of these limit sets is a maximal invariant set in $V^0$. For $\dot{x}=f_0(x)$, this union is simply the network $\Sigma_0$. For symmetry-breaking perturbations of $f_0$ it contains, but does not coincide with, a nonwandering set $\Omega_\lambda(\Sigma^\star)$ of trajectories that remain close to $\Sigma^\star$, the suspension of horseshoes accumulating on $\Sigma^\star$. The goal of this article is to investigate the larger limit set that contains nontrivial hyperbolic subsets and
attracting  limit cycles with long  periods in $V^0$. This is what Gonchenko \emph{et al} \cite{Gonchenko2007} call a \emph{strange attractor}: an attracting limit set containing nontrivial hyperbolic subsets as well as attracting periodic solutions of extremely long periods.

When $\lambda \rightarrow 0$, the horseshoes in $\Omega_\lambda(\Sigma^\star)$ lose hyperbolicity giving rise to \emph{heteroclinic tangencies} with infinitely many sinks nearby. 
A classical problem in this context is the study of heteroclinic bifurcations that lead to the
birth of stable periodic sinks -- see Afraimovich and Shilnikov \cite{Afraimovich83} and Newhouse \cite{Newhouse74, Newhouse79}. 
When we deal
with a heteroclinic tangency of the invariant manifolds, the 
description of all solutions that lie near the cycle for all time becomes more difficult. The
problem of a \emph{complete description} is unsolvable: the source of the difficulty is that arbitrarily small perturbations of any
differential equation with a quadratic homo/heteroclinic tangency (the simplest situation) may lead
to the creation of new tangencies of higher order, and to the birth of degenerate periodic orbits --- Gonchenko \cite{Gonchenko2012}.  

Large-scale invariant sets of planar Poincar\'e maps vary discontinuously in size under
small perturbations. Global bifurcations of observable sets, such as the emergence of attractors
or metamorphoses of their basin boundaries, are easily detected numerically and regularly described. However, in the example described in \cite{LR2}, the global bifurcation from a neighbourhood of $\Sigma^\star$ to $V^0$ is still a big mistery. The present paper contributes to a better understanding of the transition between uniform hyperbolicity (Smale horseshoes with infinitely many slabs) and the emergence of heteroclinic tangencies in a dissipative system close the symmetry.

\subsection*{Framework of the paper}
This paper is organised as follows. In section~\ref{object} we state our main result and 
 review some of our recent  results related to the object of study, after some basic definitions  given in section \ref{preliminaries}. 
 The coordinates and other notation  used in the rest of the article are presented in  section~\ref{localdyn}, where we also  obtain a geometrical description of the way the flow transforms a curve of initial conditions lying across the stable manifold of an equilibrium. 
In section \ref{hyperbolicity}, we investigate the limit set that contains nontrivial hyperbolic subsets and we explain how the horseshoes in $\Omega_\lambda(\Sigma^\star)$ lose hyperbolicity, as $\lambda \rightarrow 0$. 
The first obstacle towards hyperbolicity is the emergence of tangencies and the existence of thick suspended Cantor sets near the network. In section \ref{sec tangency}, we prove that there is a sequence of parameter values $\lambda_i$ accumulating on $0$ such that the flow of $f_{\lambda_i}$ has  heteroclinic tangencies and thus infinitely many attracting periodic trajectories. We include in section \ref{Conclusion} a short conclusion about the results.

\section{Preliminaries}
\label{preliminaries} 
Let $f$ be a $C^1$ vector field on $\RR^{n}$ with flow given by 
 the unique solution  $x(t)=\varphi(t,x_{0})\in \RR^{n}$ of $\dot{x}=f(x)$ and $ x(0)=x_{0}$.
Given two equilibria $p$ and $q$, an $m$-dimensional \emph{
heteroclinic connection }from $p$ to $q$, denoted 
$[p\rightarrow q]$, is an $m$-dimensional connected flow-invariant 
manifold contained in $W^{u}(p)\cap W^{s}(q)$.  There may be more than one trajectory connecting $p$ and $q$.

Let $\mathcal{S=}\{p_{j}:j\in \{1,\ldots,k\}\}$ be a finite ordered set of
mutually disjoint invariant saddles.
Following Field \cite{Field}, we say that there is a 
{\em heteroclinic cycle }associated to $\mathcal{S}$ if 
\begin{equation*}
\forall j\in \{1,\ldots,k\},W^{u}(p_{j})\cap W^{s}(p_{j+1})\neq
\emptyset \pmod k .
\end{equation*}
Sometimes we refer to the equilibria defining the heteroclinic cycle as \emph{nodes}. A \emph{heteroclinic network} is a finite connected union of heteroclinic cycles.  Throughout this article, all nodes will be hyperbolic; the dimension of the local unstable manifold of an equilibria $p$ will be called the \emph{Morse index} of $p$. 

In a three-dimensional manifold, a \emph{Bykov cycle} is a heteroclinic cycle associated to two hyperbolic saddle-foci with different Morse indices, in which the one-dimensional manifolds coincide and the two-dimensional invariant manifolds have a transverse intersection.  It arises as a bifurcation of codimension 2 and it is also called by \emph{$T$--point}.

Let $\mu$ denote a measure on a smooth manifold $M$ locally equivalent to the Lebesgue measure on charts.
 Given $x \in M$, let
$$
\omega(x) = \bigcap_{T > 0} \overline{\{\varphi(t,x) :  t \ge T\}}
$$
denote the $\omega$-limit set of the solution through $x$. If $X\subset M$ is a compact and flow-invariant subset, we let $\mathcal{B}(X) = \{x \in M : \omega(x) \subset X\}$ denote the \emph{basin of attraction} of $X$.  A compact invariant subset $X$ of $M$ is a \emph{Milnor attractor} if
$\mu(\mathcal{B}(X)) > 0$ and for any proper compact invariant subset $Y$ of $X$, $\mu(\mathcal{B}(X)\smallsetminus \mathcal{B}(Y)) > 0$.

\section{The object of study}
\label{object}
\subsection{The organising centre}
The starting point of the analysis is  a  differential equation $\dot{x}=f_0(x)$ on the unit sphere $\EU^3 =\{X=(x_1,x_2,x_3,x_4) \in \RR^4: ||X||=1\}$
where $f_0: \EU^3 \rightarrow \mathbf{T}\EU^3$ is a $C^1$ vector field  with the following properties:
\bigbreak
\begin{enumerate}
\renewcommand{\labelenumi}{(P{\theenumi})}
\item\label{P1} 
The vector field $f_0$  is equivariant under the action of $ \ZZ_2 \oplus \ZZ_2$ on 
$\EU^3$ induced by the two linear maps on $\RR^4$:
$$
\gamma_1(x_1,x_2,x_3,x_4)=(-x_1,-x_2,x_3,x_4)  \qquad \text{and} \qquad \gamma_2(x_1,x_2,x_3,x_4)=(x_1,x_2,-x_3,x_4).
$$
\item\label{P2} 
The set $Fix( \ZZ_2 \oplus \ZZ_2)=\{x \in \EU^3:\gamma_1 x=\gamma_2 x = x \}$ consists of two
equilibria $\vv=(0,0,0,1)$ and $\ww=(0,0,0,-1)$ that  are hyperbolic saddle-foci, 
 where:
\begin{itemize}
\item 
the eigenvalues of $df_0(\vv)$ are
$-C_{\vv } \pm \alpha_{\vv }i$ and $E_{\vv }$ with $\alpha_{\vv } \neq 0$, $C_{\vv }>E_{\vv }>0$ 
\item 
the eigenvalues of $df_0(\ww)$ are
$E_{\ww } \pm \alpha_{\ww } i$ and  $-C_{\ww }$ with $\alpha_{\ww } \neq 0$, $C_{\ww }>E_{\ww }>0$.
\end{itemize}
\medbreak
\item\label{P3} 
The flow-invariant circle $Fix(\zg{1})=\{x \in \EU^3:\gamma_1 x = x \}$ consists of the two equilibria 
$\vv$ and $\ww$,  a source and a sink,
respectively, and two heteroclinic trajectories 
 from $\vv$ to $\ww$ that we denote by $[\vv \rightarrow \ww]$. 
 \medbreak

\item\label{P4} 
The  $f_0$-invariant sphere $Fix(\zg{2})=\{x \in \EU^3:\gamma_2 x = x \}$ consists of the two equilibria $\vv$ and $\ww$, 
and a two-dimensional heteroclinic connection from $\ww$ to $\vv$.
Together with the connections in (P\ref{P3}) this forms a  heteroclinic network that we denote by $\Sigma^0$.
\medbreak
\item\label{P6} 
For sufficiently small open neighbourhoods $V$ and $W$ of $\vv$ and $\ww$, respectively, given any trajectory $\varphi$ going once from $V$ to $W$, if one joins the starting point of $\varphi$ in $\partial V$ to the end point in $\partial W$ by a line segment, one obtains a closed curve that is linked to  $\Sigma^0$ (figure \ref{orientations}).

\setcounter{lixo}{\value{enumi}}
\end{enumerate}

\begin{figure}
\begin{center}
\includegraphics[height=4.6cm]{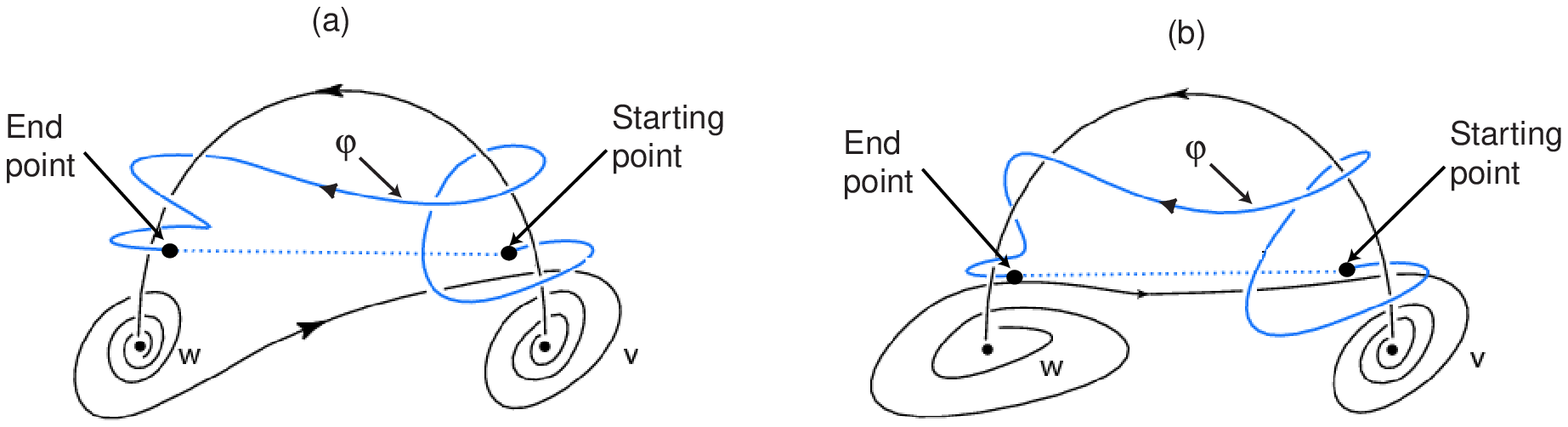}
\end{center}
\caption{\small There are two different possibilities for the geometry of the flow around a Bykov cycle  depending on the direction trajectories turn around the heteroclinic connection $[\vv \rightarrow \ww]$. 
We assume here that all trajectories   turn in the same direction   near $\vv$ and near $\ww$ as  drawn on the left (a).
When the endpoints of the trajectory are joined, the closed curve is linked to the cycle.
The case where trajectories turn in opposite directions around the connection is shown in (b). 
In this case, when the endpoints of the trajectory are joined, the closed curve may not be linked to the cycle.} 
\label{orientations}
\end{figure}

Condition (P\ref{P6}) means that the curve $\varphi$ and the cycle $\Sigma^0$ cannot be separated by an isotopy.  This property is persistent under perturbations: if it holds for the organising centre $f_0$, then it is still valid for vector fields near it, as long as the heteroclinic connection remains. 
An explicit example of a family of differential equations where this assumption is valid is constructed in Rodrigues and Labouriau \cite{LR2}.

\subsection{The heteroclinic network of the organising centre}

The heteroclinic connections in the network $\Sigma^0$ are contained in fixed point subspaces satisfying the hypothesis (H1) of Krupa and Melbourne \cite{KM1}. Since the inequality $C_\vv C_\ww >E_\vv E_\ww$ holds, the Krupa and Melbourne stability criterion  \cite{KM1}  may be applied to $\Sigma^0$ and we have: 

\begin{proposition}\label{propNetworkIstStable}
Under conditions (P\ref{P1})--(P\ref{P4})
the heteroclinic network $\Sigma^0$  is asymptotically stable. 
\end{proposition}

In particular, we obtain:

\begin{corollary}
The basin of attraction of the heteroclinic network $\Sigma^0$ whose vector field satisfies (P\ref{P1})--(P\ref{P4}) has positive Lebesgue measure.
\end{corollary}

Proposition~\ref{propNetworkIstStable}  implies that there exists an open neighbourhood $V^0$ of the network $\Sigma^0$ such that every trajectory starting in $V^0$ remains in it for all positive time and is forward asymptotic to the network. 
The neighbourhood may be taken to have its boundary transverse to the vector field $f_0$.

The fixed point hyperplane defined by $Fix(\zg{2})=\{(x_1, x_2, x_3, x_4) \in \EU^3: x_3=0\}$ divides $\EU^3$ in two flow-invariant connected components, preventing jumps between the two cycles in $\Sigma^0$.
Due to the $\zg{2}$--symmetry, trajectories whose initial condition lies outside the invariant subspaces will approach  one of the cycles in positive time. 
Successive visits to both cycles require breaking this symmetry.

\subsection{Bykov cycles}

Since $\vv$ and $\ww$ are hyperbolic equilibria, then any vector field close to it in the $C^1$ topology still has two equilibria  $\vv$ and $\ww$  with eigenvalues satisfying (P\ref{P2}). The dimensions of the local stable and unstable manifolds of $\vv$ and $\ww$ do not change.
 If  we  retain the symmetry $\gamma_1$, then the one-dimensional connection of (P\ref{P3}) remains, as it takes place in the flow-invariant circle $Fix(\zg{1})$, but generically when the symmetry  
 $\gamma_2$ is broken, the two dimensional heteroclinic connection is destroyed, since the fixed point subset $Fix(\textbf{Z}_2(\zg{2}))$ is no longer flow-invariant. 
Generically,  the invariant two-dimensional manifolds  meet transversely at two  trajectories,
 for small $\zg{1}$-equivariant perturbations of the vector field.
We denote these heteroclinic connections by $[\ww \rightarrow \vv]$
 This gives rise to a network $\Sigma^*$ consisting of four copies of the simplest heteroclinic cycle between two saddle-foci of different Morse indices, where one heteroclinic connection is structurally stable and the other is not. This cycle is called a \emph{Bykov cycle}. 
 Note that the networks $\Sigma^0$ and $\Sigma^*$ are not of the same nature. 

Transversality ensures  that the neighbourhood $V^0$ is  still  positively invariant for vector fields $C^1$ close to  $f_0$  and contains the network $\Sigma^*$.
Since the closure of $V^0$ is compact and positively invariant it contains  the $\omega$-limit sets of all its trajectories. The union of these limit sets is a maximal invariant set in $V^0$.
For $f_0$, this is the cycle $\Sigma^0$, by Proposition~\ref{propNetworkIstStable}, whereas for symmetry-breaking perturbations of $f_0$ it contains $\Sigma^*$ but does not coincide with it.
From now on, our aim is to obtain information on this invariant set.

A  systematic study of the  dynamics in a  neighbourhood  of the  Bykov cycles in $\Sigma^*$ was carried out in Aguiar {\sl et al}~\cite{ACL NONLINEARITY},  
 Labouriau and Rodrigues \cite{LR} and Rodrigues \cite{Rodrigues3};
 we proceed to review these  local results.
In  the next section we will discuss some global aspects of the dynamics. 

Given a Bykov cycle $\Gamma$ involving $\vv$ and $\ww$, let $U_\vv$ and $U_\ww \subset V^0$  be disjoint neighbourhoods of these points.
 Let $S_p$ and $S_q$ be local cross-sections of $f_\lambda$ at  two points $p$ and $q$ in the connections $[\vv\rightarrow\ww]$ and 
 $[\ww\rightarrow\vv]$,  respectively, with $p, q\not\in U_\vv\cup U_\ww$.
 Saturating the cross-sections by the flow, one obtains two flow-invariant tubes joining $U_\vv$ and $U_\ww$ containing the connections in their interior. 
 We call the union of these tubes, $U_\vv$ and $U_\ww$ a \emph{tubular neighbourhood} 
 $V^\Gamma$ of the Bykov cycle.

 With these conventions we have:
 
 \begin{theorem}[{Labouriau and Rodrigues}\cite{LR}, 2012]
\label{teorema T-point switching}
If a vector field $f_0$
satisfies (P\ref{P1})--(P\ref{P6}), then the following properties are satisfied generically by  vector fields in an open neighbourhood of $f_0$  in the space of 
$\zg{1}$--equivariant vector fields of class $C^1$ on $\EU^3$:
\begin{enumerate}
\item\label{item0}
there is a heteroclinic network $\Sigma^*$ consisting of four Bykov cycles involving two equilibria 
$\vv$ and $\ww$, two heteroclinic connections $[\vv\rightarrow\ww]$ and two heteroclinic connections
$[\ww\rightarrow\vv]$;
\item \label{item1} 
the only heteroclinic connections from $\vv $ to $\ww $ are those in the Bykov cycles and there are no homoclinic connections;
\item \label{item6}
any tubular neighbourhood $V^\Gamma$ of a Bykov cycle $\Gamma$ in $\Sigma^*$ contains points not lying on $\Gamma$ whose trajectories remain in $V^\Gamma$ for all time;
\item \label{item4} 
any tubular neighbourhood $V^\Gamma$ of a Bykov cycle $\Gamma$ in $\Sigma^*$ contains at least one $n$-pulse heteroclinic connection $[\ww\to\vv]$;
\item \label{item5}
given a  cross-section $S_q\subset V^\Gamma$ at a point  $q$ in $[\ww\rightarrow\vv]$, there exist sets of points such that the dynamics of the first return to  $S_q$ is  uniformly hyperbolic and conjugate to a full shift over a finite number of symbols. These sets  accumulate on the cycle.
\end{enumerate}
\end{theorem}

Notice that  assertion~\eqref{item4} of Theorem~\ref{teorema T-point switching}
 implies the existence of a bigger network:
beyond the original transverse connections $[\ww\rightarrow\vv]$, there exist infinitely many subsidiary heteroclinic connections turning around the original Bykov cycle. 
We will restrict our study to one cycle.

It is a folklore result that a hyperbolic invariant set of a $C^2$--diffeomorphism has zero Lebesgue measure -- see Bowen \cite{Bowen75}. However, since the authors of \cite{LR} worked in the $C^1$ category, this chain of horseshoes might have positive Lebesgue measure as the ``fat Bowen horseshoe'' described in \cite{Bowen}. Rodrigues \cite{Rodrigues3} proved that this is not the case:
\medbreak

\begin{theorem}[Rodrigues \cite{Rodrigues3}, 2013]
\label{zero measure}
Let ${V}^\Gamma$  be a tubular neighbourhood  of one of the Bykov cycles $\Gamma$ of 
Theorem~\ref{teorema T-point switching}.
Then in any cross-section $S_q\subset V^\Gamma$ at a point  $q$ in $[\ww\rightarrow\vv]$ the set of initial conditions in $S_q \cap V^\Gamma$ that do not leave $V^\Gamma$ for all time has zero Lebesgue measure.
\end{theorem}

It follows from Theorem~\ref{zero measure}  that  the shift dynamics does not trap most trajectories in the neighbourhood of the cycle. In particular, the cycle cannot be Lyapunov stable (and therefore cannot be asymptotically stable).

One astonishing property of  the heteroclinic network $\Sigma^*$  is the possibility of shadowing  $\Sigma^*$ by the property called \emph{ switching}: any infinite sequence of pseudo-orbits defined by admissible heteroclinic connections  can be shadowed, as we proceed to define.

A \emph{path} on $\Sigma^* $ is an  infinite
sequence $s^k=(c_{j})_{j\in \NN}$ of heteroclinic connections 
$c_{j}=[A_{j}\rightarrow B_{j}]$ in $\Sigma^* $ such that 
 $A_{j}, B_{j} \in \{\vv, \ww\}$ and $B_{j}=A_{j+1}$. 
 Let  $U_\vv, U_\ww\subset V^0$  be  neighbourhoods of  $\vv$ and $\ww$. 
For each heteroclinic connection in $\Sigma^*$, consider a point $p$ on it and a small neighbourhood $U_p\subset V^0$ of $p$. We assume that all these neighbourhoods  are pairwise disjoint.

The trajectory $\varphi(t,q)$ of $\dot x=f_\lambda(x)$,  \emph{follows} the  path
$s^k=(c_{j})_{j\in \NN}$ within a given set of neighbourhoods as above, if
there exist two monotonically increasing sequences of times 
$(t_{i})_{i\in \NN}$ and $(z_{i})_{i\in \NN}$
such that for all $i \in \NN$, we have $t_{i}<z_{i}<t_{i+1}$ and:
\begin{itemize}
\item
$\varphi(t,q)\subset V^0$ for all $t\ge 0$;
\item
$\varphi(t_{i},q) \in U_{A_{i}}$ and $\varphi (z_{i},q)\in U_{p}$, $p\in c_i$ and
\item
 for all $t\in (z_{i},z_{i+1})$, $\varphi (t,q)$ does not visit the neighbourhood  of any other node except $A_{i+1}$.
\end{itemize}
There is \emph{switching} near $\Sigma^*$ if     for each  path  there is a trajectory that follows it within  every set of neighbourhoods as above.  

\begin{theorem}[Aguiar {\sl et al} \cite{ACL NONLINEARITY}, 2005]
\label{teorema switching}
 There is  switching on the network $\Sigma^*$.
\end{theorem}

The solutions that realise  switching lie for all positive time  in  the union of   tubular neighbourhoods $V^\Gamma$ of all cycles $\Gamma\subset\Sigma^*$.
Hence we may adapt the proof of  Theorem~\ref{zero measure}  to obtain:

\medbreak
\begin{corollary}
The  switching of Theorem \ref{teorema switching} is realised by a set of initial conditions with zero Lebesgue measure. 
\end{corollary}

\subsection{Non-hyperbolic dynamics}

When the symmetry $\gamma_2$ is broken, the dynamics changes dramatically, as can be seen in the next result:

\begin{theorem}
\label{teorema tangency}
If a vector field $f_0$ satisfies (P\ref{P1})--(P\ref{P6}), then in any open neighbourhood of $f_0$ in the space of $\zg{1}$-equivariant vector fields of class $C^1$ on $\EU^3$, there are vector fields $f_*$
whose flow  has a heteroclinic tangency between $W^u(\ww)$ and $W^s(\vv)$ in $V^0$. 
\end{theorem} 

Theorem~\ref{teorema tangency} is proved in section \ref{sec tangency}.
The vector fields  $f_*$ will be obtained in a generic one-parameter unfolding $f_\lambda$ of $f_0$, for which we will find a sequence of $\lambda_i$ converging to zero such that the flow of $f_{\lambda_i}$ has the required property.
When $\lambda_i \rightarrow 0$, these tangencies accumulate on the transverse connections.
Persistent tangencies in a dissipative diffeomorphism are related to the coexistence of infinitely many sinks and sources \cite{Newhouse74,Newhouse79}. Moreover, any parametrised family of diffeomorphisms going through a heteroclinic tangency associated to a dissipative cycle must contain a sequence of H\'enon-like families  \cite{Colli}.  Hence, for $\lambda \approx 0$, return maps to appropriate domains close to the tangency are conjugate to H\'enon-like maps and thus: 

\begin{corollary}
\label{attractors}
Suppose a vector field $f_0$ satisfies (P\ref{P1})--(P\ref{P6}).
Then in the space of $\zg{1}$-equivariant vector fields of class $C^3$ on $\EU^3$,
there is a set $\mathcal{C}$ of vector fields accumulating on $f_0$ such that  all $f_\lambda \in \mathcal{C}$ possess infinitely many strange (coexisting) attractors in $V^0$ .
\end{corollary}

\begin{proof}
In Ovsyannikov and  Shilnikov \cite{OS}, it is shown that there are small perturbations of $f_0$ with a periodic solution as close as desired to the cycle whose stable and unstable manifolds are tangent. By Colli \cite{Colli}, the result follows.
Note that although the statement in \cite{Colli} asks for $C^\infty$ perturbations, in \cite{BonattiEcia} the coexistence of strange attractors only requires $C^3$. 
\end{proof}

 \section{Local dynamics near the network}
  \label{localdyn}
 In order to obtain results on $\zg{1}$--equivariant  vector fields $C^1$ close to  $f_0$ we study  the bifurcation  of a generic one-parameter family of differential equations $\dot{x}=f_\lambda(x)$ on the unit sphere $\EU^3 =\{X=(x_1,x_2,x_3,x_4) \in \RR^4: ||X||=1\}$.
The unfolding $f_\lambda: \EU^3 \rightarrow \mathbf{T}\EU^3$ of $f_0$  is a family of $C^1$ vector fields  with the following properties:
\begin{enumerate}
\renewcommand{\labelenumi}{(P{\theenumi})}
\setcounter{enumi}{\value{lixo}}
\item\label{PartialSymmetry}
For each $\lambda$ the vector field  $f_\lambda$ is $\zg{1}$--equivariant.
\medbreak
\item\label{P4.5}
There are two equilibria $\vv$ and $\ww$ satisfying (P\ref{P2}) and (P\ref{P3}).
\medbreak
\item\label{P5}  
For $\lambda \neq 0$, the local two-dimensional manifolds $W^u(\ww)$ and $W^s(\vv)$ intersect transversely at two  trajectories.
Together with the connections in (P\ref{P3}) this forms a Bykov  heteroclinic network that we denote by $\Sigma^\lambda$.
\setcounter{lixo}{\value{enumi}}
\end{enumerate}

In order to describe the dynamics around the Bykov cycles, we start by introducing  local coordinates near the saddle-foci $\vv$ and $\ww$ and we define some terminology that will be used in the rest of the paper.
Since by assumption (P\ref{P2}) we have $C_\vv \neq E_\vv$ and $C_\ww \neq E_\ww$, then by Samovol's Theorem \cite{Samovol}, the vector field $f_\lambda$ is $C^1$--conjugate to its linear part around each saddle-focus  --- see also Homburg and Sandstede \cite{HS} (section 3.1).  
In cylindrical coordinates $(\rho ,\theta ,z)$ the linearization at $\vv$ takes the form
$$
 \dot{\rho}=-C_{\vv }\rho \qquad  \dot{\theta}=1  \qquad  \dot{z}=E_{\vv }z
$$
and around $\ww$ it is given by:
$$
\dot{\rho}=E_{\ww }\rho \qquad \dot{\theta}= 1  \qquad \dot{z}=-C_{\ww }z .
$$
 \begin{figure}
\begin{center}
\includegraphics[height=6cm]{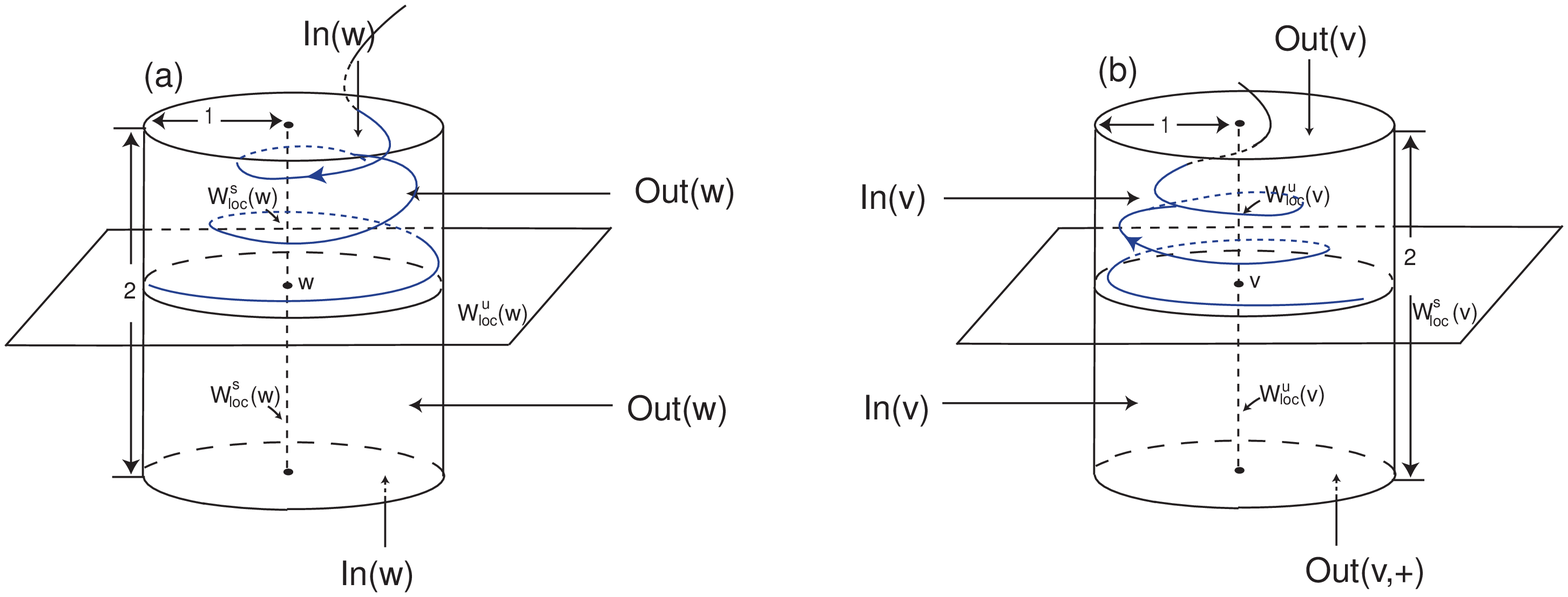}
\end{center}
\caption{\small  Cylindrical neighbourhoods of the saddle-foci $\ww$ (a) and $\vv$ (b). }
\label{neigh_vw}
\end{figure}

We  consider cylindrical neighbourhoods  $V$ and $W$  in ${\EU}^3$ of $\vv $ and $\ww $, respectively, of radius $\varepsilon>0$ and height $2\varepsilon$ --- see figure \ref{neigh_vw}.
After a linear rescaling of the local variables, we  may take $\varepsilon=1$.
Their boundaries consist of three components: the cylinder wall parametrized by $x\in \RR\pmod{2\pi}$ and $|y|\leq 1$ with the usual cover $ (x,y)\mapsto (1 ,x,y)=(\rho ,\theta ,z)$ and two discs (top and bottom of the cylinder). We take polar coverings of these discs $(r,\varphi )\mapsto (r,\varphi , \pm 1)=(\rho ,\theta ,z)$ whith $0\leq r\leq 1$ and $\varphi \in \RR\pmod{2\pi}$ and use the following terminology, as in  figure \ref{neigh_vw}:
\begin{itemize}
\item
$In(\vv)$, the cylinder wall of $V$,  consists of points that go inside $V$ in positive time;
\item
$Out(\vv)$, the top and bottom of $V$,  consists of points that go inside $V$ in negative time;
\item
$In(\ww)$, the top and bottom of $W$,  consists of points that go inside $W$ in positive time;
\item
$Out(\ww)$,  the cylinder wall  of $W$,  consists of points that go inside $W$ in negative time.
\end{itemize}
The flow is transverse to these sets and moreover the boundaries of $V$ and of $W$ may be written as the closures of the disjoint unions $In(\vv) \cup Out (\vv)$ and  $In(\ww) \cup Out (\ww)$, respectively.
The trajectories of  all points $(x,y)$ in $In(\vv) \backslash W^s(\vv)$, leave $V$ at $Out(\vv)$ at
\begin{equation}
\Phi_{\vv }(x,y)=\left(|y|^{\delta_\vv},-\frac{\ln |y|}{E_\vv}+x\right)=(r,\phi)
\qquad \mbox{where}\qquad 
\delta_\vv=\frac{C_{\vv }}{E_{\vv}} > 1 \  .
\label{local_v}
\end{equation}
Similarly, points $(r,\phi)$ in $In(\ww) \backslash W^s(\ww)$, leave $W$ at $Out(\ww)$ at
\begin{equation}
\Phi_{\ww }(r,\varphi )=\left(-\frac{\ln r}{E_\ww}+\varphi,r^{\delta_\ww}\right)=(x,y)
\qquad \mbox{where}\qquad 
\delta_\ww=\frac{C_{\ww }}{E_{\ww}} >1 \ .
\label{local_w}
\end{equation}

We will denote by $W^u_{loc}(\vv)$ the portion of $W^u(\vv)$  that goes from $\vv$ up to $In(\ww)$ not intersecting the interior of $W$.
Similarly,  $W^s_{loc}(\vv)$ is the portion of $W^s(\vv)$ outside $W$ that goes directly  from $Out(\ww)$ into $\vv$, and  $W^u_{loc}(\ww)$ and $W^s_{loc}(\ww)$ connect $\ww$ to $\partial V$, not intersecting the interior of $V$.

The flow sends points in $Out(\vv)$ near $W^u_{loc}(\vv)$ into $In(\ww)$ along the connection 
$[\vv\rightarrow\ww]$. We will assume that this map $\Psi_{\vv \rightarrow \ww}$ is the identity, this is compatible with hypothesis (P\ref{P6}); nevertheless all the results follow  if $\Psi_{\vv \rightarrow \ww}$ is either a uniform contraction or a uniform expansion.
We  make the convention that one of the  connections $[\vv\rightarrow\ww]$
links points with $y>0$ in $V$ to points with $y>0$ in $W$.
There is also a well defined transition map 
$ \Psi_{\ww \rightarrow \vv}:Out(\ww)\longrightarrow In(\vv)$ that will be discussed later.

 By (P\ref{P5}), the manifolds $W^u(\ww)$ and $W^s(\vv)$ intersect transversely for $\lambda \neq 0$. For $\lambda$ close to zero, we are
assuming that $W^s(\vv)$ intersects the wall $Out(\ww)$ of the cylinder $W$ on a closed curve  represented
in figure~\ref{elipse} by an ellipse --- this is the expected
unfolding from the coincidence of the invariant manifolds of the equilibria.

\begin{figure}
\begin{center}
\includegraphics[height=12cm]{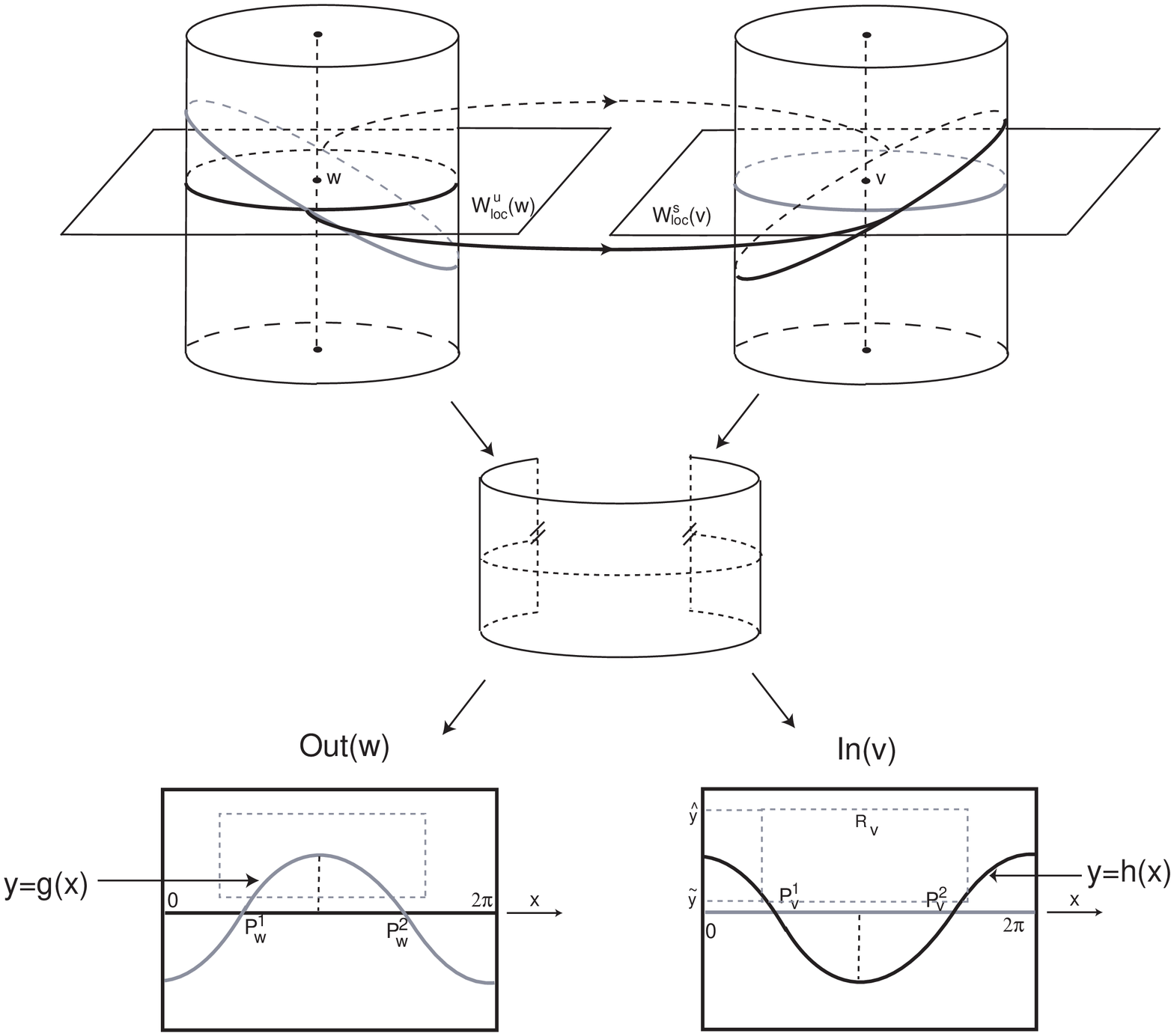}
\end{center}
\caption{\small For $\lambda$ close to zero, we are assuming that $W^s(\vv)$ intersects the wall $Out(\ww)$ of the cylinder $W$ on a closed curve -- this is the expected unfolding from the coincidence of the invariant manifolds of the equilibria.}
\label{elipse}
\end{figure}

\bigbreak

From the geometrical behaviour  of the local transition  maps (\ref{local_v}) and (\ref{local_w}), we need some definitions: a \emph{ segment }$\beta $ 
on $In(\vv)$ is a smooth regular parametrized curve 
$\beta :[0,1)\rightarrow In(\vv)$ that meets $W^{s}_{loc}(\vv )$ transversely at the point $\beta (1)$ only and such that, writing $\beta (s)=(x(s),y(s))$, both $x$ and $y$ are monotonic functions of $s$ -- see figure \ref{Structures} (a).

A \emph{ spiral } on a disc $D$  around a point $p\in D$ is a curve 
$\alpha :[0,1)\rightarrow D$
satisfying $\dpt \lim_{s\to 1^-}\alpha (s)=p$ and such that if
$\alpha (s)=(r(s),\theta(s))$ is its expressions in
polar coordinates around $p$ then the maps $r$ and $\theta$ are monotonic, and 
$\lim_{s\to 1^-}|\theta(s)|=+\infty$.

Consider a cylinder $C$ parametrized by a covering $(\theta,h )\in  \RR\times[a,b]$,
with $a<b\in\RR$ where $\theta $ is periodic.
A \emph{helix} on the cylinder $C$ 
\emph{accumulating on the circle} 
$h=h_{0}$ is a curve
$\gamma :[0,1)\rightarrow C$
such that its coordinates $(\theta (s),h(s))$ 
satisfy $ \lim_{s\to 1^-}h(s)=h_{0}$, $\lim_{s\to 1^-}|\theta (s)|=+\infty$ and the maps $\theta$ and $h$ are monotonic.

Using these definitions and the expressions \eqref{local_v} and \eqref{local_w} for $\Phi_{\vv }$ and $\Phi_{\ww }$  we get:

\begin{proposition}[Aguiar \emph{et al} \cite{ALR}, 2010]
\label{Structures}
When (P\ref{P6}) holds, then
a segment on $In(\vv)$ is mapped by $\Phi _{\vv}$ into 
 a spiral on $Out(\vv)$ around $W^u_{loc}(\vv)\cap Out(\vv) $.
 This spiral is mapped by $ \Psi_{\vv \rightarrow \ww}$  into another spiral around $W^s_{loc}(\ww)\cap In(\vv)$,
 which is mapped
by $\Phi _{\ww}$ into a helix on $Out(\ww)$ accumulating on the circle  $Out(\ww) \cap W^{u}(\ww)$.
\end{proposition}

\section{Hyperbolicity}
\label{hyperbolicity}
In this section we show that the hyperbolicity of Theorem~\ref{teorema T-point switching} only holds when we restrict our attention to trajectories that remain near the  cycles in the network.
The construction also indicates how  the geometrical content of
Theorems~\ref{teorema T-point switching}, \ref{zero measure} and \ref{teorema switching} is obtained.

 Let $(P_\ww^1,0)$ and $(P_\ww^2,0)$ with $0<P_\ww^1<P_\ww^2<2\pi$ be the coordinates of the two points in 
 $W^u_{loc}(\ww)\cap W^s_{loc}(\vv)\cap Out(\ww)$
 where the connections  $[\ww \rightarrow \vv]$ meet $Out(\ww)$, as in figure~\ref{elipse}.
 Analogously, let $(P_\vv^1,0)$ and $(P_\vv^2,0)$ be the coordinates of the two corresponding points  in $W^u_{loc}(\ww)\cap W^s_{loc}(\vv)\cap In(\vv)$ where $[\ww \rightarrow \vv]$ meets $In(\vv)$, with the convention that $(P_\ww^j,0)$ and $(P_\vv^j,0)$ are on the same trajectory for $j=1,2$.
 
 For small $\lambda>0$ we may write $W^s(\vv)\cap Out(\ww)$ as the graph
  of a smooth function  $y=g(x)$, with $g(P_\ww^j)=0$, $j=1,2$. 
  Similarly, $W^u(\ww)\cap In(\vv)$ is the graph
  of a smooth function  $y=h(x)$, with $g(P_\vv^j)=0$, $j=1,2$. 
  For definiteness, we number the points in the connections to have 
   $g^\prime(P_\ww^1)>0$. Hence   $g^\prime(P_\ww^2)<0$, and 
$h^\prime(P_\vv^1)<0$,  $h^\prime(P_\vv^2)>0$. 

\begin{proposition}
\label{propLines}
For    the first hit map $\eta= \Phi _{\ww}\circ \Psi_{\vv \rightarrow \ww}\circ \Phi _{\vv}:In(\vv)\seta Out(\ww)$, we have:
\begin{enumerate}
\renewcommand{\theenumi}{\roman{enumi}}
\renewcommand{\labelenumi}{{\theenumi})}
\item\label{horizontal}
any horizontal line segment $[a,b]\times\{y_0\}\subset In(\vv)$ is mapped by $\eta$ into a horizontal 
line segment $[c,d]\times \{y_0^\delta\}\subset Out(\ww)$, with $\delta>1$; 
\item\label{vertical}
any vertical line segment $\{x_0\}\times[0,y_0]\subset In(\vv)$ is mapped by $\eta$ into a helix accumulating on the circle $Out(\ww)\cap W^u_{loc}(\ww)$; 
\item\label{intervals}
given $x_0$, there are positive constants $a<b\in\RR$ and a sequence of  intervals \\
$\mathcal{I}_n= \{x_0\} \times [e^{-2n\pi/K}e^a, e^{-2n\pi/K}e^b] $
such that $\eta\left(\mathcal{I}_n\right)$ crosses $W^s_{loc}(\vv)\cap Out(\ww)$ twice transversely;
\item\label{disjoint}
if $K>1$ then the intervals $\mathcal{I}_n$ are disjoint.
\end{enumerate}
\end{proposition}

\begin{proof}
Assertion \ref{horizontal}) is   immediate  from the expression of 
$\eta=\Phi_{\ww } \circ \Psi_{\vv \rightarrow \ww} \circ \Phi_{\vv }$ in coordinates:
$$
\eta(x,y)=\left(x-K \ln y , y^{\delta} \right)
\qquad\mbox{where}\qquad
K= \frac{C_\vv+E_\ww}{E_\vv E_\ww} > 0
\qquad\mbox{and}\qquad
\delta=\delta_\vv \delta_\ww>1 
$$
and assertion \ref{vertical}) follows from Proposition~\ref{Structures}.

For \ref{intervals}) let $y_*$ be the maximum value of $g(x)$ and let $m\in\ZZ$ such that
$x_0-P_\ww^2\le \frac{K}{\delta}\ln y_*+2m\pi$. 
We take $a=(x_0-P_\ww^1-2m\pi)/K$ and  $b=(x_0-P_\ww^2-2m\pi)/K$.
Then $a<b$, the second coordinate of $\eta(x_0,e^b)$ is less than $y_*$,
its first coordinate is $P_\ww^2+2m\pi$, and that of $\eta(x_0,e^a)$ is $P_\ww^1+2m\pi$.
Hence the curve $\eta(\mathcal{I}_0)$ goes across the graph of $g(x)$ that corresponds to  
$W^s(\vv)$ as in figure~\ref{figPropLines}.
The first coordinates of  the end points of  $\eta(\mathcal{I}_n)$ for the other intervals 
are $P_\ww^1+2(m+n)\pi$ and $P_\ww^2+2(m+n)\pi$ and their second coordinates are also less than $y_*$, so each curve $\eta(\mathcal{I}_n)$ also crosses the graph  of $g(x)$ transversely.

If $K>1$, and since $P_\ww^1-P_\ww^2<2\pi<2K\pi$, then 
$x_0-P_\ww^2-2m\pi-2K\pi<x_0-P_\ww^1-2m\pi$ and hence $b-2(n+1)\pi/K<a-2n\pi/K$, 
implying that  $\mathcal{I}_n\cap\mathcal{I}_{n+1}=\emptyset$.
\end{proof}

\begin{figure}
\begin{center}
\includegraphics[height=7cm]{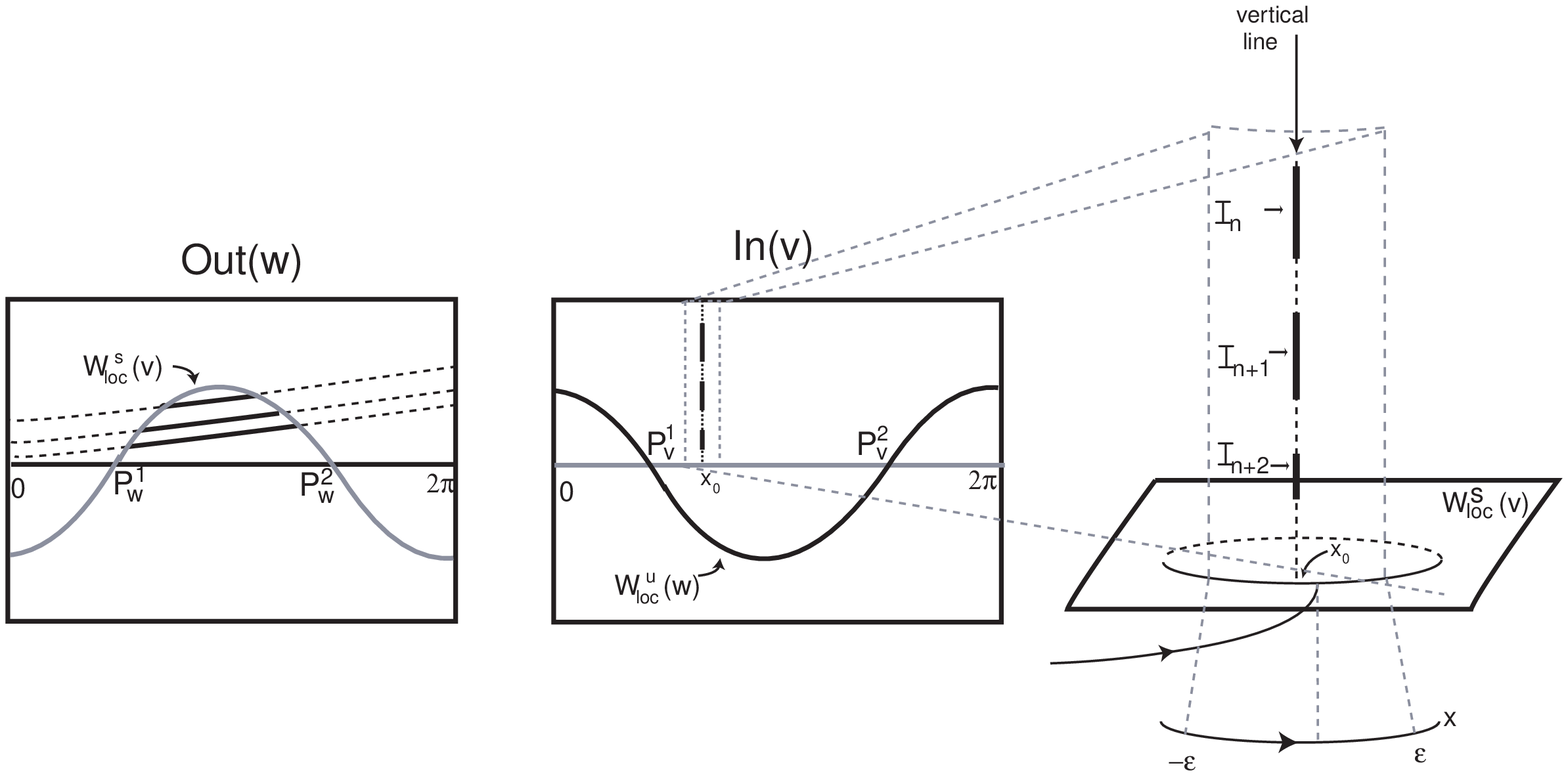}
\end{center}
\caption{\small Proof of \ref{intervals}) in Proposition~\ref{propLines}: the transition map $\eta$ sends the intervals  $\mathcal{I}_n$ into curves that cross $W^s_{loc}(\vv)$ transversely. }
\label{figPropLines}
\end{figure}

We are interested in the images of rectangles in $In(\vv)$ under iteration by
 the first return map to $In(\vv)\backslash W^s_{loc}(\vv)$ given by $ \zeta= \Psi_{\ww \rightarrow \vv}   \circ \eta$.

Consider a rectangle $R_\vv=[P_\vv^1-\tau,P_\vv^1+\tau]\times[\tilde{y},\hat{y}]\subset In(\vv)$ with $\tau>0$ small and $0<\tilde{y}<\hat{y}<1$.  
This  is mapped by $\eta$ into a strip whose boundary  consists of  two horizontal line segments and two pieces of helices. 
A calculation similar to that in Proposition~\ref{propLines} shows that there are many choices of $\tilde{y}$ and $\hat{y}$ for which the strip crosses $W^s_{loc}(\vv)$ transversely near $P_\ww^1$.
This strip is then mapped by $\Psi_{\ww \rightarrow \vv}$ into $In(v)$  crossing $W^s_{loc}(\vv)$ transversely. 
The final strip $\zeta(R_\vv)=\Psi_{\ww \rightarrow \vv}\circ\eta(R_\vv)$ remains close to 
$W^u_{loc}(\ww)$.
Hence the effect of $\zeta$ is to stretch
$R_\vv$ in the vertical direction and map it  with the stretched direction approximately parallel to 
$W^u_{loc}(\ww)$.
Since $R_\vv$ has been chosen to contain a piece of $W^u_{loc}(\ww)$, then $\zeta(R_\vv)$ will cross $R_\vv$.
Repeating this for successive disjoint intervals $[\tilde{y},\hat{y}]$ gives rise to horseshoes.

This is the idea of the proof of the local results of Theorem~\ref{teorema T-point switching}:
in a neighbourhood of the connection point $(P_\vv^1,0)$, one  finds infinitely many disjoint rectangles, each one containing  a Cantor set of points whose orbits under $\zeta$  remain in the Cantor set, and hence return to a neighbourhood of $(P_\vv^1,0)$ for all future iterations.
In the gaps between these rectangles one finds another set of disjoint rectangles that are first mapped by $\zeta$ into a neighbourhood of the other connection point  $(P_\vv^2,0)$.
Repeating the construction near the second connection one obtains the switching of 
Theorem~\ref{teorema switching}.

It can also be shown that the first return map $\zeta$ is uniformly hyperbolic at the points in the rectangle $R_\vv$, with the choices of $\tilde{y},\hat{y}$ above. This means that at each of these points there is a well defined contraction direction and this is the main tool in the proof of 
Theorem~\ref{zero measure}.

Since all this takes place inside a positively invariant neighbourhood $V^0$,
 it would be natural to try to extend this reasoning to larger rectangles in $In(\vv)$.
We finish this section explaining why this fails.

Consider a rectangle $R_\vv$ as above, containing points of the Cantor set. 
Since $\eta$ expands vertical lines, the local unstable manifolds  of these points forms a lamination on $R_\vv$ whose sheets  are approximately vertical. 
Now, take a larger  rectangle $\widehat{R}_\vv$ with $y\in[\tilde{y},\bar{y}]$, $\bar{y}>\hat{y}$, so as to have the maximum of the curve $y=g(x)$ lying in $\eta(\widehat{R}_\vv)$, and increasing $\tau$ if necessary.
The sheets of the lamination still follow the vertical direction in the enlarged rectangle, and their image by $\eta$ is approximately a helix on $Out(\ww)$.
Changing the value of the bifurcation parameter $\lambda$ moves  the graph $y=g(x)$ (the stable manifold of $\vv$) but does not affect the map $\eta$.
Hence, by varying $\lambda$ we can get a sheet of the lamination tangent to $y=g(x)$,
say, $W^u(x_0,y_0)$ for some point $(x_0,y_0)$ in the Cantor set as in figure~\ref{lamina}.
This breaks the hyperbolicity, since it means that $W^u(x_0, y_0)$ is tangent to $W^s(\vv)$.
This phenomenon has been studied by Gonchenko {\sl et al.} \cite{Gonchenko2007} --- it corresponds to a decrease in topological entropy as in figure~\ref{transition1}.
As the images of the rectangles move down, each time one of them crosses a rectangle a sequence of saddle-node bifurcations starts, together with a period-doubling cascade, as on the right hand side of figure~\ref{transition1}.

A more rigorous construction will be made in the next section.
\begin{figure}
\begin{center}
\includegraphics[height=5cm]{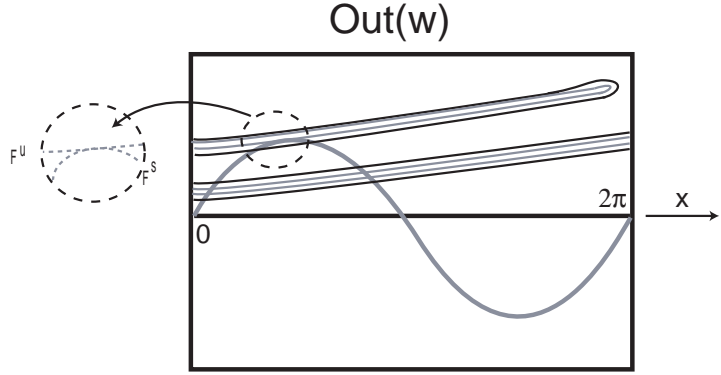}
\end{center}
\caption{\small Near the point of maximum height of $W^s_{loc}(\vv)\cap Out(\ww)$  the unstable manifold of some point is tangent to $W^s(\vv)$. By varying $\lambda$ this point may be taken to be in the Cantor set of points
that remain near the cycle. }
\label{lamina}
\end{figure}

\begin{figure}
\begin{center}
\includegraphics[height=8.7cm]{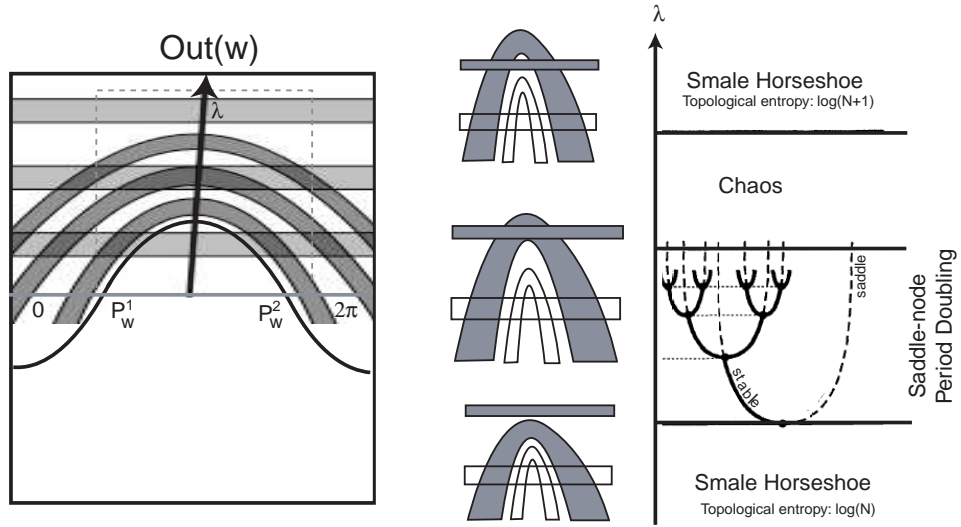}
\end{center}
\caption{\small When $\lambda$ decreases, the Cantor set of points  of the horseshoes
that remain near the cycle is losing topological entropy, as the set loses hyperbolicity.
This happens when   the unstable manifold of some point in the Cantor set is tangent to $W^s(\vv)$.}
\label{transition1}
\end{figure}

\section{Heteroclinic Tangency}
\label{sec tangency}

In this section we show how to find values of the bifurcation parameter $\lambda$ for which $W^u(\ww)$ is tangent to $W^s(\vv)$.

\begin{proof}[Proof of Theorem~\ref{teorema tangency}]
With  the notation  of section~\ref{hyperbolicity}, 
the two points $(P_\ww^1,0)$ and $(P_\ww^2,0)$ divide the closed curve $y=g(x)$
 where $W_{loc}^s(\vv)$ intersects $Out(\ww)$ in two components, corresponding to different signs of the second coordinate.
 With the conventions  of section~\ref{hyperbolicity}, we get $g(x)>0$ for  $x\in\left(P_\ww^1,P_\ww^2\right)$.
Then the region in $Out(\ww)$ delimited by  $W_{loc}^s(\vv)$ and  $W_{loc}^u(\ww)$ between $P_\ww^1$ and $P_\ww^2$ gets mapped by $ \Psi_{\ww \rightarrow \vv}$ into the $y<0$ component of 
$In(\vv)$, while all other points  in  $Out(\ww)$ with positive second coordinates, are mapped ito the $y>0$ component of $In(\vv)$ as in  figure~\ref{Psiwv}.
The maximum value  of the coordinate $y$ for the curve
$W_{loc}^s(\vv)\cap Out(\ww)$ is of the order of $\lambda$, attained at some point 
$(x,y)\approx (x_m,\lambda)$ with $P_\ww^1<x_m<P_\ww^2\pmod{2\pi}$.

Consider now the closed curve where $W_{loc}^u(\ww)$ intersects $In(\vv)$.
For small values of $\lambda$ this is approximately an ellipse, crossing $W_{loc}^s(\vv)$ at the two points $(P_\vv^1,0)$ and $(P_\vv^2,0)$, see figure~\ref{elipse}.
With the conventions  of section~\ref{hyperbolicity}, this is the graph $y=h(x)$ with
 $h(x)>0$ for  $x\in\left(P_\vv^2,P_\vv^1\right)$.
In particular, the portion of this curve that lies in the upper half of $In(\vv)$, parametrised by $(x,y)$, $y>0$, may be written as the union of two segments  $\sigma_1$ and $\sigma_2$ that meet at the point where the coordinate $y$ attains its maximum value on the curve.
Without loss of generality, let $(x_*,\lambda)$ be the coordinates  in $In(\vv)$ of this point, whith 
$P_\vv^2<x_*<P_\vv^1\pmod{2\pi}$.

\begin{figure}
\begin{center}
\includegraphics[height=4.5cm]{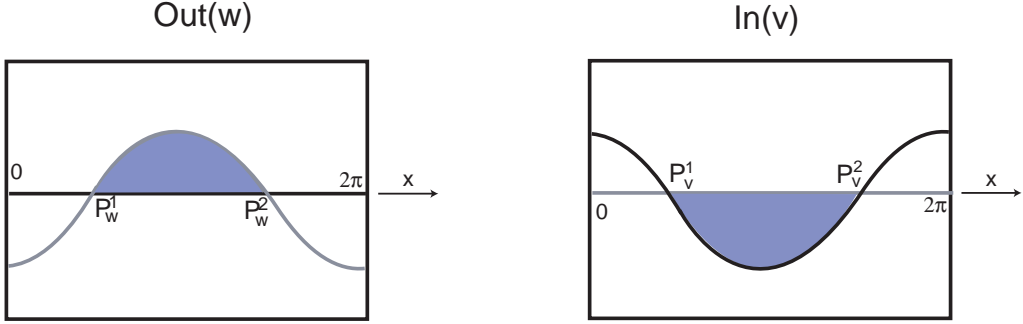}
\end{center}
\caption{\small The transition map $ \Psi_{\ww \rightarrow \vv}$ sends the shaded area in $Out(\ww)$, between  $W_{loc}^s(\vv)$ and  $W_{loc}^u(\ww)$ and with positive second coordinate, into the shaded area in $In(\vv)$ with negative second coordinate. Conventions: black line stands for  $W_{loc}^u(\ww)$, grey line is $W_{loc}^s(\vv)$. The rest of the upper part of $Out(\ww)$ is mapped into the upper part of $In(\vv)$.}
\label{Psiwv}
\end{figure}

By Proposition~\ref{Structures}, the image of each segment $\sigma_j$ by  $\eta$ is a helix on 
$Out(\ww)$ accumulating on $W_{loc}^u(\ww)$.
Hence, the curve $\eta(\sigma_1\cup \sigma_2)$ is a double helix.
The projection of this curve into $W_{loc}^u(\ww)$ is regular at all points, except for a fold at 
$\eta(x_*,\lambda)=(x_*-K\ln\lambda, \lambda^\delta)=(x(\lambda),y(\lambda))$, as in figure~\ref{homoclinic1}.
As $\lambda$ decreases to zero,  the first coordinate $x(\lambda)$ of $\eta(x_*,\lambda)$ tends to infinity, hence the point $\eta(x_*,\lambda)$ makes infinitely many turns around the cylinder $Out(\ww)$. 

On the other hand,   $y(\lambda)= \lambda^\delta$ with $\delta>1$, so $y(\lambda)$ decreases to zero faster than $\lambda$, the maximum height of the curve $W_{loc}^s(\vv)\cap Out(\ww)$.
Therefore, given any small $\lambda_0>0$, there exists a positive $\lambda_1<\lambda_0$ such that 
$x(\lambda_1)=x_m$ and moreover $\eta(x_*,\lambda_1)=(x(\lambda_1),y(\lambda_1))$
lies in the region in $Out(\ww)$ between $W_{loc}^s(\vv)$ and  $W_{loc}^u(\ww)$ 
that gets mapped into the lower part of $In(\vv)$.
For $\lambda=\lambda_1$, the points on the curve $W_{loc}^u(\ww)\cap In(\vv)$ close to $(x_*,\lambda_1)$ are also mapped by $\eta$ into the lower half of $Out(\ww)$.

Furthermore, there exists a positive $\lambda_2<\lambda_1$ such that  $P_\ww^2<x(\lambda_2)<P_\ww^1\pmod{2\pi}$ and hence $\eta(x_*,\lambda_2)=(x(\lambda_2),y(\lambda_2))$ is mapped by  
$ \Psi_{\ww \rightarrow \vv}$ into the upper part of $In(\vv)$. 
Again, for $\lambda=\lambda_2$, points on the curve $W_{loc}^u(\ww)\cap In(\vv)$ close to $(x_*,\lambda)$ return to the upper part of $Out(\ww)$. 

Therefore, for some $\lambda_*$,
with $\lambda_2<\lambda_*<\lambda_1$,  the image of the curve $W_{loc}^u(\ww)\cap In(\vv)$ by the first return map to $In(\vv)$ is tangent to $W_{loc}^s(\vv)\cap In(\vv)$, given in local coordinates by $y=0$. This completes the proof of  Theorem~\ref{teorema tangency} --- given any $\lambda_0>0$ we have found a positive $\lambda_*<\lambda_0$ such that
$W^u(\ww)$ is tangent to $W^w(\vv)$  for  $f_{\lambda_*}$. 
 \end{proof}

\section{Conclusion}
\label{Conclusion}
For the present study, we have used the symmetry $\gamma_1$ and its flow-invariant fixed-point subspace to ensure the persistence of the
connections $[\vv \rightarrow \ww]$ of one-dimensional manifolds.
The symmetry is not essential for our exposition but its existence makes it more natural.
In particular, the networks we describe are persistent within the class of differential equations with the prescribed symmetry.
\bigbreak

As a global structure, the transition from in the dynamics from $\dot{x}=f_0(x)$ to  $\dot{x}=f_\lambda(x)$, $\lambda \approx 0$ is intriguing and has not
always attracted appropriate attention. 
There is a  neighbourhood $V^0$ with positive  Lebesgue  measure that is positively invariant for the flow of $\dot{x}=f_\lambda(x)$. 
For $\lambda=0$ all trajectories approach the network $\Sigma^0$. 
For $\lambda > 0$ and for a sufficiently small  tubular neighbourhood $V^\Gamma \subset V^0 $ of any of  the Bykov cycles in $\Sigma^\lambda$, almost all trajectories might return to 
$V^\Gamma$ but they do not necessarily remain there for all future time. 
Trajectories that remain in $V^\Gamma$ for all future time form an infinite set of suspended horseshoes with zero Lebesgue measure.

For a fixed $\lambda>0$, when we take a larger tubular  neighbourhood $V^\Gamma$ of a Bykov cycle $\Gamma$, the suspended horseshoes lose hyperbolicity.  While small symmetry-breaking terms generically destroy the attracting cycle $\Sigma^0$, there will still be an attractor lying close to the original cycle.
This is the main point of  section~\ref{hyperbolicity}: when local invariant manifolds are extended, they develop tangencies, which explain the attractivity. The existence of primary heteroclinic tangencies is proved in section~\ref{sec tangency}.

Heteroclinic tangencies give rise to attracting periodic trajectories of large periods and small basins of attraction, appearing in large (possibly infinite) numbers.
Heteroclinic tangencies also create new tangencies near them in phase space and for nearby parameter values.
We know very little about the geometry of these strange attractors, we also do not know the size and the shape of their basins of attraction.

 When $\lambda \rightarrow 0$, the infinite number of periodic sinks lying close to the network of Bykov cycles will approach the ghost of $\Sigma^0$. 
For the parameter values where we observe heteroclinic tangencies, each Bykov cycle possesses infinitely many sinks whose basins of attractions have positive three-dimensional  Lesbesgue measure. 
The attractors must lie near $\overline{W^s(\vv) \cup W^u(\ww)}$ and they  collapse into $\Sigma^0$ as $\lambda \rightarrow 0$. 
A lot more needs to be done before the subject is well understood.

\begin{figure}
\begin{center}
\includegraphics[height=4.5cm]{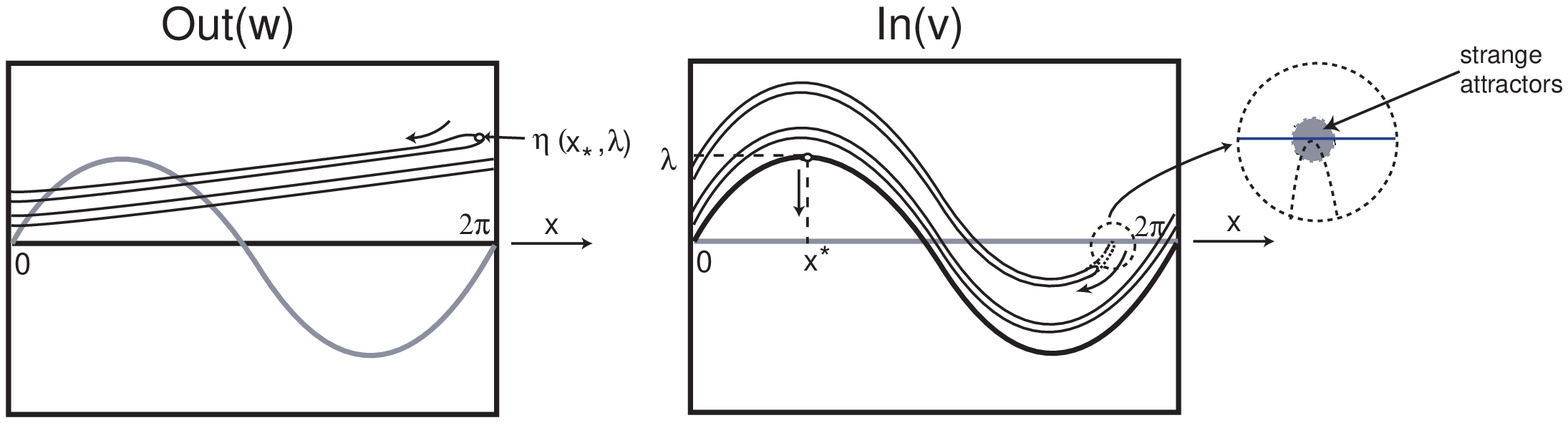}
\end{center}
\caption{In any open neighbourhood of $f_0$ in the space of $\zg{1}$--equivariant vector fields of class $C^1$ on $\EU^3$, there is a sequence of vector fields $f_{\lambda_i}$ accumulating of $f_0$ whose flow  has a heteroclinic tangency between $W^u(\ww)$ and $W^s(\vv)$.}
\label{homoclinic1}
\end{figure}

\bigbreak
\textbf{Acknowledgements:} The authors would like to thank Maria Carvalho for helpful discussions.

\end{document}